\documentclass{amsart}[12pt]

\usepackage{amsmath, amsthm, enumerate, amsfonts, mathrsfs, xypic, color}
\xyoption{all}

\newtheorem{theorem}{Theorem}
\newtheorem{proposition}{Proposition}

\newtheorem{lemma}{Lemma}

\newcommand{\tr}{\operatorname{tr}}

\title{Causal geometries and third-order ordinary differential equations}
\author{Jonathan Holland}
\author{George Sparling}

\begin{document}
\maketitle
\begin{abstract}
We discuss contact invariant structures on the space of solutions of a third order ordinary differential equation.  Associated to any third-order differential equation modulo contact transformations, Chern \cite{Chern} introduced a degenerate conformal Lorentzian metric on the space $J^2$ of 2-jets of functions of one variable.  When the scalar invariant of W\"{u}nschmann \cite{WunschmannThesis} vanishes, the degenerate metric descends to a proper conformal Lorentzian metric on the space of solutions.  In the general case, when the W\"{u}nschmann invariant is not zero, we define the notion of a causal geometry, and show that the space of solutions supports one.   The W\"{u}nschmann invariant is then related to the projective curvature of the indicatrix curve cut out by the causal geometry in the projective tangent space. When the W\"{u}nschmann vanishes, the causal geometry is then precisely the sheaf of null geodesics of the Chern conformal structure.   We then introduce a Lagrangian and associated Hamiltonian from which the degenerate conformal Lorentzian metric are constructed.  Finally, necessary and sufficient conditions are given for a rank three degenerate conformal Lorentzian metric in four dimensions to correspond to a third-order differential equation.
\end{abstract}

\nocite{*}

\section{Introduction}
The purpose of the present paper is to study the geometry of third-order ordinary differential equations: equations of the form
$$y''' = F(x,y,y',y'').$$
By setting $p=y'$ and $q=y''$, the solutions of this equation coincide with curves in the second jet space $J^2$ with coordinates $(x,y,p,q)$ that are everywhere tangent to the contact distribution annihilated by the one-forms
$$\theta_1 = dy-pdx,\quad \theta_2=dp-qdx$$
and that are also annihilated by the one-form
$$\omega = dq-F(x,y,p,q)dx.$$
A contact transformation $\Phi:J^2\to J^2$ is a local diffeomorphism that preserves the contact filtration on $J^2$, meaning that
$$\Phi^*\theta_1\equiv 0 \pmod{\theta_1},\quad\Phi^*\theta_2\equiv 0\pmod{\theta_1,\theta_2}.$$
Contact transformations act on the set of differential equations by composition.  The paper is concerned with structures that are invariant with respect to transformations of this form.  An alternative characterization of third-order equations and contact transformations relying only on the first jet space $J^1$ and the contact structure associated to $\theta_1$ will be given in Section \ref{ThirdOrderDifferentialEquations}.

The equivalence problem for third-order ordinary differential equations has a long history.  It was first studied under the class of point transformations by \'{E}lie Cartan \cite{CartanThirdOrder}.  In 1940, S. S. Chern \cite{Chern} focused on the equivalence problem under contact transformations.  A certain scalar invariant of the structure, called the W\"{u}nschmann invariant, divides the family of third-order equations into two classes.  Those with vanishing W\"{u}nschmann invariant admit a natural conformal Lorentzian structure on the space of solutions.  For Chern, this conformal structure presented itself in the form of a normal $SO(3,2)$ Cartan connection on the solution space or, what is the same, a conformal Lorentzian metric on that space.  A geometrical description of this conformal structure was presented much later by Fritelli, Kozameh, and Newman \cite{FKN}, who showed that the general third-order equation with vanishing W\"{u}nschmann invariant could be obtained by considering one-parameter families of null hypersurfaces in a 3-dimensional space with a conformal Lorentzian metric.  More directly, as discussed in Section \ref{ConformalStructure}, this case can be understood also in terms of the null geodesic spray on the bundle of null rays: null hypersurfaces then being related by means of an envelope construction.

Chern also determined all of the contact invariants of the general third-order equation in which the W\"{u}nschmann is nonzero.  This was presented in the modern language of bundles and connections by Sato and Yoshikawa \cite{SatoYoshikawa}, who in addition clarified the geometrical interpretation of the structure as a normal $SO(3,2)$ Cartan connection on the space $J^2$.  Nurowski and the second author later showed the existence of a conformal $O(3,3)$ structure on a certain fiber bundle over $J^2$ (see \cite{NurowskiConformal}).  This $O(3,3)$ structure is of Fefferman \cite{Fefferman} type, in the sense of Graham \cite{Graham}, if and only if the W\"{u}nschmann invariant vanishes.  Godli\'{n}ski \cite{Godlinski} then proved that the associated normal $SO(4,4)$ Cartan connection included the Chern--Sato--Yoshikawa connection as its $\mathfrak{o}(3,2)$ part.

The structure of Chern--Sato--Yoshikawa can doubtless be understood directly in terms of the geometry of the space of solutions of the differential equation.  However, whereas when the W\"{u}nschmann invariant vanishes, there is a standard geometry underlying the presence of certain connections---namely a conformal Lorentzian metric on the space of solutions---when the W\"{u}nschmann invariant is nonzero, such an underlying geometry appears to be missing.  The main task of the present article is to supply the missing geometry and to examine its precise relationship with these constructions.

Invariantly associated to the structure on $J^2$ is a degenerate conformal Lorentzian metric (see Nurowski \cite{NurowskiConformal}) whose degeneracy is in the direction of the total derivative vector field coming from the differential equation.  The second task of the present article is to show that this degenerate conformal structure arises naturally from elementary geometrical constructions on a certain curve (the {\em polar curve}) in the projective cotangent bundle of the space of solutions.  The W\"{u}nschmann invariant itself is precisely the projective curvature of this curve.  Finally, the paper introduces a new conformal invariant of fourth order in the metric that gives a complete geometrical characterization of degenerate metrics that arise in this manner from third-order differential equations.

The first innovation of the paper is that of an {\em incidence relation} on the space of solutions, described in Section \ref{Incidence}.  An infinitesimal or linearized version of this idea is implicit in W\"{u}nschmann's \cite{WunschmannThesis} investigations into Monge equations of the second degree (see also \cite{Liebmann} and \cite{Chern}).  When the W\"{u}nschmann invariant vanishes, two solutions are incident if and only if they lie on the same null geodesic.  When the W\"{u}nschmann is nonzero, the incidence relation still defines a decent structure in a sense that is axiomatized in Section \ref{CausalStructures}: roughly, the sheaf of curves defining the incidence relation is envelope-forming.  Such a family of envelope-forming curves is dubbed a {\em causal geometry}, the terminology suggested by an affinity with structures that typically arise in the study of hyperbolic partial differential equations.  The space of incidence curves in the solution space corresponds naturally to the points of the 1-jet space $J^1$.  When the incidence curves through a point are linearized at that point, the resulting cone in the tangent space resembles the null cone associated to a Lorentzian structure.  The null cone projects to a curve in the projective tangent space, called the {\em indicatrix} curve, borrowing terminology from optics \cite{Arnold}.

The indicatrix gives rise to a Lagrangian in a natural manner that can be written down in terms of the general solution of the differential equation, as described in Section \ref{Lagrangian}.  The Lagrangian is a function on the tangent bundle which is homogeneous of degree two with respect to the scalar homothety of the bundle.  It is not fully contact-invariant, but its locus of zeros in the projective tangent bundle is invariant, and coincides with the indicatrix.  Null geodesics---extremals of the Lagrangian along which the Lagrangian vanishes identically---are precisely the incidence curves.  The resulting structure is a Finsler \cite{Finsler} analog of conformal Lorentzian geometry.

The Lagrangian is in addition regular at every point of the indicatrix, and therefore gives rise to a Hamiltonian on the cotangent bundle, which is described in Section \ref{Hamiltonian}. The zero locus of the Hamiltonian inside the projective cotangent space is the polar curve of the indicatrix.  The total space of the indicatrix or its polar curve defines a 4-dimensional bundle over the space of solutions, and the projective Hamiltonian spray defines a projective vector field on this bundle.  The 3-dimensional quotient space under the flow of the vector field inherits a natural contact form from the cotangent bundle.  The resulting space is contactomorphic to $J^1$, the space of 1-jets in the plane, and on it the polar curves descend to a path geometry that defines a third-order differential equation in the manner described in Section \ref{ThirdOrderDifferentialEquations}.

The entire procedure sketched here is summarized in the theorem.

\begin{theorem}\label{locallyisomorphic}
There is a natural local isomorphism between the set of third-order equations under contact equivalence and the set of isomorphism classes of causal geometries.
\end{theorem}

The rest of the paper is devoted to studying the degenerate rank three conformal Lorentzian metrics $g$ on a four-dimensional space $N$.  The fundamental invariant is
$$\Gamma = g \wedge \mathscr{L}_V g \wedge\mathscr{L}_V^2 g \wedge \mathscr{L}_V^3 g \wedge \mathscr{L}_V^4 g \in \wedge^5 S^2\ker (V\lrcorner) \cong S^2(TN/V)$$
where $V$ is the degenerate direction.  Here $\ker (V\lrcorner)$ is the space of one-forms annihilated by $V$, and each of the Lie derivatives $\mathscr{L}_V^kg$ lies in the symmetric square $S^2\ker (V\lrcorner)$.  The space $S^2\ker(V\lrcorner)$ is six-dimensional, and its fifth exterior power is naturally isomorphic to the symmetric square $S^2(TN/V)$ of the quotient of the tangent bundle of $N$ by the vertical direction $V$.

\begin{theorem}
The degenerate Lorentzian metric $g$ on the four-manifold $N$ arises (locally) from a third-order differential equation if and only if either
\begin{enumerate}
\item $\Gamma$ is nonzero and the classical adjoint of $\Gamma$ vanishes identically (equivalently, $\Gamma$ has rank $1$).  In this case, there exists a natural conformal isometry of $N$ with $J^2$ equipped with its invariant degenerate metric coming from a third-order differential equation with non-zero W\"{u}nschmann invariant.
\item $\mathscr{L}_Vg$ is proportional to $g$.  In this case, the W\"{u}nschmann invariant vanishes and there exists  a conformal isometry of $N$ with $J^2$ equipped with its invariant degenerate metric coming from a third-order differential equation that is natural up to a gauge transformation of $J^2$.
\end{enumerate}
\end{theorem}

\section{Third-order differential equations}\label{ThirdOrderDifferentialEquations}
A third-order differential equation under contact equivalence can be conveniently regarded as the following data.
\begin{enumerate}
\item A three-dimensional contact manifold $J^1$.
\item A generic three-parameter family of (unparameterized) contact curves in $J^1$.
\end{enumerate}
The contact structure on $J^1$ can represented as a one-form $\theta$ which is only invariantly defined up to scale.  By Darboux' theorem, there exist coordinates $(x,y,p)$ on $J^1$ such that $\theta=dy-pdx$.  A contact curve is a curve whose tangent annihilates $\theta$ at each point.  The family of curves is generic if at each given point of $J^1$ and tangent direction $v$ at $x$ annihilating $\theta$ there exists a unique curve through $x$ with tangent along $v$.  These curves are identified with the (prolongation of) solutions of the differential equation.  Two such structures are locally equivalent if there is a local diffeomorphism of $J^1$ to itself that preserves the contact structure and sends one system of curves to the other.

Given a third-order differential equation, it is clear how to generate such a structure by prolongation (see, for instance, \cite{Olver}), and the resulting structure depends only on the contact-equivalence class of the differential equation, by B\"{a}cklund's theorem.  Conversely, suppose we have chosen coordinates $(x,y,p)$ on $J^1$ such that $\theta=dy-pdx$.  The distinguished class of curves is of the form
\begin{align*}
x&=\chi(s;a,b,c)\\
y&=\psi(s;a,b,c)\\
p&=\pi(s;a,b,c)
\end{align*}
where $a,b,c$ are the three parameters defining a curve in the class, and $s$ is an evolution parameter of the curve.  There is a gauge freedom in selecting the parameterization $s$ of the curve, and so this freedom is eliminated this by imposing the condition $dx/ds=1$  (that is, by effectively taking $x$ itself to be the parameter).  The contact relation takes the form $p=dy/dx$.  Imposing this relation and differentiating $\psi$ three times gives the system of equations
\begin{align*}
y&=\psi(x;a,b,c)\\
y'&=\psi_x(x;a,b,c)\\
y''&=\psi_{xx}(x;a,b,c)\\
y'''&=\psi_{xxx}(x;a,b,c).
\end{align*}
Solving the first three equations for $a,b,c$ in terms of $x,y',y''$ and substituting the result into the third equation gives a third-order differential equation $y'''=F(x,y,y',y'')$.

Associated to the jet space $J^1$ there is naturally a fibration $\pi:J^2\to J^1$. This is the subbundle of the projective tangent bundle $\mathbb{P}TJ^1$ of $J^1$ given as the annihilator of $\theta$: $J^2=\theta^\perp$.  Specifically, $J^2$ is given fiberwise by
$$J^2_x = \left\{v\in \mathbb{P}T_xJ^1 \mid v\lrcorner\theta = 0\right\}.$$
It is a four-dimensional space fibered over $J^1$ with $S^1$ fibers. The space $J^2$ defined here supports the following contact-invariant structure, independently of the differential equation.  This characterization is the four-dimensional analog of structures studied in five dimensions by the authors in \cite{DoubrovHollandSparling}.

\begin{lemma}\label{J2}
\begin{enumerate}
\item There exists a natural filtration
$$T^1 \subset T^2 \subset  T^3 \subset  T^4 =  TJ^2$$
of the tangent bundle of $J^2$.  Here $T^1$ is the vertical distribution for the fibration $J^2\to J^1$, $T^2$ is a tautological bundle of $2$-planes, and $T^3$ is the annihilator of the pullback of $\theta$.
\item On sections, $[\Gamma(T^1),\Gamma(T^2)]=[\Gamma(T^1),\Gamma(T^3)]=\Gamma(T^3)$ and $[\Gamma(T^3),\Gamma(T^3)]=\Gamma(T^4)$.
\end{enumerate}
\end{lemma}

Specifically, $T^2$ is the tautological $2$-plane bundle whose fiber at a point $(x,u)\in J^2_x \subset \mathbb{P}T_xJ^1$ consists of all vectors $v$ such that $\pi_*v$ is in the direction of $u$. In terms of the $(x,y,p)$ coordinates on $J^1$, any vector field of the form $\partial/\partial x + p\partial/\partial y + q\partial/\partial p$ annihilates the contact form $\theta=dy-pdx$.  Therefore this $q$ defines a fiber coordinate for the fibration $J^2\to J^1$ that allows the vector fields generating $T^2J^2$ to be expressed as $X = \partial/\partial q$, the vertical vector field for the fibration, and $\partial/\partial x + p\partial/\partial y + q\partial/\partial p$.  The lemma follows by taking commutators.

\begin{lemma}\label{J2characterization}
Any two four-manifolds equipped with this structure are locally isomorphic: a filtration $T^1\subset T^2\subset T^3\subset T^4$ on the tangent bundle, such that  $[\Gamma(T^1),\Gamma(T^2)]=[\Gamma(T^1),\Gamma(T^3)]=\Gamma(T^3)$ and $[\Gamma(T^3),\Gamma(T^3)]=\Gamma(T^4)$.
\end{lemma}
\begin{proof}
Let $M$ be a manifold equipped with such a filtration and let $\theta$ be a nonvanishing one-form annihilating $T^3M$. Let $N$ be the 3-manifold obtained by passing to the (locally defined) quotient modulo the flow of $T^1M$.  The distribution $T^3M$ is Lie derived along $T^1M$, and so descends to a distribution of $2$-planes on $N$.  Nowhere is this distribution Frobenius integrable, and so it defines a contact structure on $N$.  By Darboux' theorem, $N$ is locally contactomorphic to $J^1$ with its standard contact structure and coordinates $(x,y,p)$.   Letting $q$ be a fiber coordinate on $M\to N$, $X=\partial/\partial q$ generates $T^1M$ and $\theta=dy-p\,dx$.  Plane subbundles of $T^3$ on which $\mathscr{L}_X$ maps surjectively onto $T^3$ are all related by a change in the fiber coordinate $q$.  Indeed, as $X$ commutes with $\partial/\partial p$, the latter vector field does not lie in $T^2$.  The bundle $T^2$ must contain a solution $Y$ of $\mathscr{L}_X Y = \partial/\partial p.$  One such solution is $Y=\partial/\partial x + p\partial/\partial y + q\partial/\partial p$, and the ambiguity, modulo $T^1$, in the solution is a transformation of the form $Y\mapsto Y + \lambda \partial/\partial p$ where $X(\lambda)=0$.  Noting that $\lambda$ is independent of $q$, this ambiguity in the choice of $Y$ can be absorbed into a change of coordinates $q\mapsto q + \mu$ where $\mu(x,y,p)$ satisfies the differential equation $\mu-\mu_x-\mu_y-\mu_p=\lambda$.  So we are free to choose the fiber coordinate $q$ so that $T^2$ is generated by $X$ and $Y=\partial/\partial x + p\partial/\partial y + q\partial/\partial p$. The coordinates $(x,y,p,q)$ now defined on $M$ establish a local diffeomorphism with $J^2$ that sends $T^1M$, $T^2M$, and $T^3M$ to $T^1J^2$, $T^2J^2$, and $T^3J^2$, respectively.
\end{proof}


The differential equation is specified in terms of a splitting of the first level of the filtration $T^1J^2\subset T^2J^2$.  This is achieved by means of a vector field $V$ of the form $V=\partial/\partial x + p\partial/\partial y + q\partial/\partial p + F(x,y,p,q)\partial/\partial q$ representing the total derivative.  While $V$ itself is not contact-invariant, the splitting direction consisting of all multiples of $V$ is. This splitting can invariantly be described in terms of the system of curves on $J^1$ that gives the differential equation.  Lying over a point $x\in J^1$, the point $u\in J^2_x$ of the fiber is by definition a projective tangent vector at $x$.   Passing through $x$ in the direction defined by $u$ is a distinguished curve of the differential equation.  This curve determines a tangent direction at every point which therefore specifies a lift to $J^2$.  The vector $V_{(x,u)}$ is the tangent direction to the lifted curve at the point $(x,u)\in J^2$.

\subsection{Conformal structure}\label{ConformalStructure}
It is possible to construct from these data a degenerate conformal Lorentzian metric $g$ on $J^2$.  The degenerate direction for the metric is $V$, and the vector field $X$ is null with respect to the metric.  In terms of the coordinates $(x,y,p,q)$ on $J^2$, this metric is given by
\begin{equation}\label{Chernmetric}
g = 2[dy-p\,dx][dq-\tfrac13 F_q\,dp+K\,dy+(\tfrac13 qF_q-F-pK)dx]-[dp-q\,dx]^2
\end{equation}
where $K=\tfrac16 V(F_q)-\tfrac19 F_q^2-\tfrac12 F_p$. In Section \ref{Metric}, the present paper constructs this metric in a manifestly contact-invariant fashion.  The (locally defined) quotient space $\mathbb{M} = J^2/V$ is a three-manifold, but the metric $g$ may not be Lie derived up to scale along $V$, and so need not pass down to the quotient.  The Lie derivative $\mathscr{L}_Vg$ is proportional to $g$ if and only if the W\"{u}nschmann invariant of the original differential equation vanishes.  The W\"{u}nschmann invariant is given by
\begin{equation}\label{Wunschmann}
W=F_y +(V-\tfrac23 F_q)K.
\end{equation}
Vanishing of this invariant is a necessary and sufficient condition for $\mathbb{M}$ to possess an invariant Lorentzian conformal structure.

Going the other way, let $\mathbb{M}$ be a conformal Lorentzian $3$-manifold.  Define $\mathbb{S}$ to be the bundle over $\mathbb{M}$ with fiber $S^1$ that, at each point $P$, consists of all null directions in $\mathbb{P}T^*_P\mathbb{M}$.  The pullback metric is degenerate in the vertical direction, and so $\mathbb S$ supports the structure of a degenerate conformal Lorentzian $4$-manifold for which the degenerate direction is a conformal Killing symmetry.  There is a canonical symplectic potential $\psi$ defined on the total space of the cotangent bundle of $\mathbb M$.  The form $\psi$ is annihilated by the scaling in the fiber, and is Lie derived up to scale, and so descends to give a form $\theta$ up to scale on $\mathbb S$.  The condition $\theta\wedge d\theta\not=0$ follows since $\psi\wedge d\psi$ on $T^*\mathbb M$ does not vanish when pulled back to nonzero sections $\mathbb M\to T^*\mathbb M$.  Indeed, in local coordinates $\mathbf{x}=(x^1,x^2,x^3)$ on $\mathbf{M}$ with fiber coordinates $\mathbf{p} = (p_1,p_2,p_3)$ on $T^*\mathbf{M}$, $\psi\wedge d\psi = (\mathbf{p}\cdot d\mathbf{x})(d\mathbf{p}\cdot d\mathbf{x}) = -\frac{1}{2}(\mathbf{p}\times d\mathbf{p})\cdot (d\mathbf{x}\times d\mathbf{x})$.  Since $d\mathbf{x}\times d\mathbf{x}$ has three linearly independent components, and $\mathbf{p}\times d\mathbf{p}$ vanishes only on vectors parallel to the generator of scalings in $T^*\mathbb{M}$,  $\psi\wedge d\psi$ does not vanish when pulled back along any section of the projective cotangent bundle, and so {\em a fortiori} it does not vanish when pulled back to $\mathbb{S}$.

The vector field $X$ is the given by the null geodesic spray in $\mathbb{S}$.  To describe this, fix a metric $g$ in the conformal class on $\mathbb M$.  On $T^*\mathbb M$ the geodesic Hamiltonian is $H=\pi^*g^{-1}(\psi,\psi)$, which gives rise to the Hamiltonian vector field $\widehat{H}$, defined by $\widehat{H}\lrcorner d\psi = dH$.  An integral curve $\mu$ of $X$ projects to a geodesic of $M$, and the fiber component of $\mu$ is the covelocity of the geodesic.  The image of $\mathbb{S}$ under the map $\mathbb{S}\to T^*M$ is the null cone at every point of $\mathbb M$.  This is everywhere tangent to the spray $\widehat{H}$, because a geodesic being initially null will always remain null.  Furthermore, $\widehat{H}$ scales quadratically in the cotangent bundle, and descends to a direction field on the projective cotangent bundle $\mathbb{P}T^*\mathbb{M}$.  This direction field is everywhere tangent to $\mathbb{S}$, and so restricts to a direction field $X$ on $\mathbb{S}$.

Passing to the (locally defined) quotient by $X$ gives the space of null geodesics---the twistor space $\mathbb{T}$.  It follows from the definition of $X$ and $\theta$ that $X\lrcorner\theta=0$ and $X\lrcorner d\theta=0$.  
The second assertion follows by pulling back $X\lrcorner d\psi = dH$ to the null cone, and using the fact that $H$ vanishes identically there.  Thus $\theta$ descends to a contact structure on the twistor space.  The fibers of $\mathbb{S}\to\mathbb{M}$ project to the distinguished curves of the twistor space.

To prove that these distinguished curves are suitably generic, we verify that $\mathbb{S}$ admits a structure satisfying the conditions of Lemma \ref{J2characterization}.  Let $T^1$ be the bundle spanned by the null geodesic spray $X$ and $T^3$ be annihilator of $\theta$.  From $\theta\wedge d\theta\not=0$, it follows that $T^3$ is not Frobenius integrable, and so $[\Gamma(T^3),\Gamma(T^3)]=\Gamma(T\mathbb{S})$.  Moreover, since $\theta$ is annihilated by $X$ and Lie derived along it, $[\Gamma(T^1),\Gamma(T^3)]=\Gamma(T^3)$.

It remains to identify $T^2$ and to show that  $[\Gamma(T^1),\Gamma(T^2)]=\Gamma(T^3)$.  Let $V$ be a nonvanishing vertical vector field for $\mathbb{S}\to\mathbb{M}$ and let $T^2=\operatorname{span}\{ X, V\}$.  One such vector field $V$ can be given in terms of the angular momentum operator $\mathbf{L}=\mathbf{p}\times \partial/\partial\mathbf{p}$.  Then $\mathbf{L}$ is tangent to the null cone, because it annihilates $H$.  On homogeneous functions of degree 0, the angular momentum factors through a scalar operator $\mathbf{L}(f) = V(f)\mathbf{p}$, which defines the vector field $V$.   Because $\theta$ is horizontal, any such vector field, being vertical, annihilates $\theta$.   Furthermore, $[X,V]\lrcorner\theta = \mathscr{L}_X (V\lrcorner\theta)=0$ as well, so $[X,V]\in T^3$. It remains only to show that $X, V, [X,V]$ are linearly independent.  Each of these vector fields is at most first order in the metric, and therefore independence follows by a calculation in normal coordinates.  In these coordinates, $X=-\mathbf{p}\cdot\partial/\partial\mathbf{x}$ and so on the one hand
$$\left[X, \mathbf{L}\right] = \mathbf{p}\times\frac{\partial}{\partial\mathbf{x}}$$
and on the other hand, on the null cone this acts as $\mathbf{p}[X,V]$ on functions homogeneous of degree zero.  Finally,
$$d\theta\left(\left[X, \mathbf{L}\right],  \mathbf{L}\right) = (\mathbf{p}\cdot\mathbf{p}) g - \mathbf{p}\otimes\mathbf{p}$$
which vanishes nowhere on the null cone.  Thus $[\Gamma(T^1),\Gamma(T^2)]=\Gamma(T^3)$.  

So modulo the explicit construction of the degenerate metric and the assertion that specifically it is the W\"{u}nschmann invariant that governs whether the degenerate metric descends to the 3-manifold, we have proven the following theorem (proven in Fritelli, Kozameh, Newman \cite{FKN} by entirely different methods):

\begin{theorem}\label{conformalLorentzian}
There is a natural local equivalence between third-order differential equations under contact transformations with vanishing W\"{u}nschmann invariant and conformal Lorentzian 3-manifolds.
\end{theorem}

\section{Causal geometries on the space of solutions}\label{CausalStructures}
This section develops the natural geometric structure associated to the space of solutions $\mathbb{M}$ to a third-order equation.  This structure reduces to a conventional Lorentzian conformal structure if and only if the W\"{u}nschmann invariant vanishes.  To motivate this discussion, the previous section establishes the basic properties of the (possibly locally defined) double fibration
$$
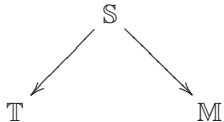
\xymatrix{
&\ar[dl]\mathbb{S}\ar[dr]&\\
\mathbb{T}&&\mathbb{M}
}$$

An individual fiber of the submersion $\mathbb{S}\to\mathbb{M}$ projects down to give a trajectory solving the differential equation in $\mathbb{T}$: this is a re-expression of the notion that $\mathbb{M}$ is a {\em space of solutions} of the differential equation.  The fibers for the other submersion $\mathbb{S}\to\mathbb{T}$ also project down to the space of solutions $\mathbb{M}$, although their precise meaning has heretofore not been identified in general.

When the W\"{u}nschmann invariant vanishes, $\mathbb{M}$ carries a natural conformal Lorentzian metric by Theorem \ref{conformalLorentzian}, and $\mathbb{S}$ is canonically identified with the null cone bundle associated to this metric.  The space $\mathbb{T}$ is then the twistor space: the quotient of $\mathbb{S}$ by the null geodesic spray.  The fibration $\mathbb{S}\to\mathbb{M}$ can be understood as the subbundle of the projective tangent bundle $\mathbb{P}T\mathbb{M}$ of null directions.  Under this correspondence, a null cone with vertex at $P\in \mathbb{M}$ corresponds to a one-parameter family of null geodesics, which in turn is identified with the trajectory defining the solution $P$.

This structure can be axiomatized in a manner that allows construction of the spaces $\mathbb{S}$ and $\mathbb{T}$, along with their natural contact structure.  The one parameter families of null geodesics starting at each point $P$ give rise to a cone in the tangent space $T_P\mathbb{M}$ with vertex at the origin.  Equivalently, such a cone is the affine cone over some curve in the projective space $\mathbb{P}T_P\mathbb{M}$.  Here locality considerations may mean that the cone may fail to close up completely, or the associated curve may have one or more singular points.  Henceforth, we shall work only near regular points of the curve.\footnote{We can localize near a point of $\mathbb{T}$ and a point of $\mathbb{M}$ simultaneously.  Thus all solutions can be assumed to be fully regular, but cannot necessarily be continued for all time.  The geometrical implication is that we may isolate a smooth part of the null cone at each point of $\mathbb{M}$, but the ``cone'' need not then close up.}

\subsection{Incidence relation and the indicatrix}\label{Incidence}
Two solutions $y_1$ and $y_2$ to the differential equation $y'''=F(x,y,y',y'')$ are said to be {\em incident} if, at some point $x$, one has
$$y_1(x)=y_2(x),\quad y_1'(x) = y_2'(x).$$
That is to say, two curves are incident if and only if the associated solution curves in the $J^1$ intersect.  In the latter interpretation, the incidence relation is manifestly contact-invariant.

The set of all solutions incident with a given solution $P\in\mathbb{M}$ cuts out a surface $N_P$ with a conical type singularity at the vertex $P$.  The generators of this cone can be described as follows.  If $Q$ is a solution incident with $P$, then the curve $N_{PQ}$ consisting of solutions $R$ incident with $P$ at the same point of $J^1$ as $Q$. (Our localization assumption implies that two distinct solutions are incident at most at a single point of $J^1$.)  The surface $N_P$ is thus ruled by the pencil of curves $N_{PQ}$ as $Q$ varies over solutions incident with $P$.

The incidence relation described here is called a {\em causal geometry}.  The following axioms are assumed to hold in a sufficiently small neighborhood of each point $P\in\mathbb{M}$:
\begin{enumerate}[{\bf {Axiom} 1.}]
\item $N_P\setminus\{P\}$ is a smooth embedded hypersurface in $\mathbb{M}$
\item The singular variety $N_P$ is ruled by a pencil of smooth curves $N_{PQ}$ from $P$ to $Q$, as $Q$ varies over $N_P\setminus\{P\}$.
\item For every $Q\in N_P\setminus\{P\}$, the curve $N_{PQ}$ on $N_P$ coincides with the curve $N_{QP}$ on $N_Q$: $N_{PQ}=N_{QP}$.
\item Let $C_P$ be the set of tangents at $P$ to the curves $N_{PQ}$ as $Q$ varies over $N_P\setminus\{P\}$.  Then $C_P$ is a regular (smooth) curve in the projective tangent space $\mathbb{P}T_P\mathbb{M}$, and does not make second order contact with any of its tangent lines.  (The curve $C_P$ is called the {\em indicatrix}.)
\item Let $Q,R\in N_P\setminus\{P\}$ be given distinct points which are mutually incident. Then the tangent plane to $N_P$ at $R$ is the same as the tangent plane to $N_Q$ at $R$: $T_RN_P=T_RN_Q$.  ({\em Envelope condition})
\end{enumerate}
Theorem \ref{causalTheorem} establishes that these axioms hold in the case of the causal geometry on the space of solutions of a third-order differential equation.  The first four axioms are fairly natural assumptions that make precise the notion that the $N_P$ should be a cone based at $P$.  The envelope condition of axiom 5, in terms of the bundle $J^1\to\mathbb{M}$, guarantees that the one-form $\theta$, whose annihilator is the 3-plane bundle that lifts the tangent planes through the vertices of the cones $N_P$, is Lie derived along the fibers of the fibration $J^2\to J^1$.

Causal geometries also appear naturally in connection with the following variational problem.  Let $L:T\mathbb{M}\setminus\{0\}\to\mathbb{R}$ be a homogeneous function of degree two (a {\em Lagrangian}), possibly only defined on an open conical subbundle of $T\mathbb{M}\setminus\{0\}$.  Suppose that for each $P\in\mathbb{M}$, the curve $C_P = \{v\in \mathbb{P}T_P\mathbb{M}\mid L(P,v) = 0\}$ is smooth.  For a curve $\gamma$ in $\mathbb{M}$, consider the energy functional
$$E[\gamma] = \frac{1}{2}\int_a^b L(\gamma(t),\gamma'(t))\,dt.$$
A null geodesic is a curve $\gamma$ that is an extremals for $E$ along which $L(\gamma(t),\gamma'(t))\equiv 0$.  The family of null geodesics for a Lagrangian defines a causal geometry if and only if the Lagrangian is regular at every point of the cone over $C_P$ in $T_PM$ for each $P$, in the sense that its vertical Hessian is nondegenerate in directions tangent to $C_P$.  In coordinates $x^i$ for $\mathbb{M}$ and fiber coordinates $\dot{x}^i$ for $T\mathbb{M}$, this is the condition that the $3\times 3$ matrix $\frac{\partial^2L}{\partial \dot{x}^i\partial \dot{x}^j}$ be nonsingular in directions tangent to the null cone.  This is also a sufficient condition for $L$ to be conserved along an extremal, and thus the null geodesics are precisely those extremals of the energy for which $L(\gamma(0),\gamma'(0))=0$.

The Euler--Lagrange system is a second order ordinary differential equation for the curve $\gamma(t)$.  By smooth dependence on initial conditions, for initial conditions lying on the hypersurface $L(P,\gamma'(0))=0$ sufficiently near the origin of the tangent space $T_P\mathbb{M}$, the null geodesics foliate a smooth hypersurface in $\mathbb{M}$, giving axioms 1--2.  Axiom 3 follows since the null geodesics are critical points of the energy under compactly supported variations, and so in particular are characterized independently of the direction of their parameterization.  Axiom 4 follows from the regularity of the Lagrangian.  Finally, for axiom 5, it is sufficient to show that, for any point $P$, and any variation $\gamma_s$ of null geodesics through $P$,
\begin{equation}\label{firstvariation}
\left.\frac{\partial L}{\partial \dot{x}^i} \frac{d\gamma^i}{ds}\right|_{s=0} = 0.
\end{equation}
This then establishes that the null cone at $P$ is tangent to the indicatrix curve at every point, which is equivalent to axiom 5.  By the Euler--Lagrange equations,
\begin{align*}
0=\frac{d}{ds}L(\gamma_s,\dot{\gamma}_s) &= \frac{\partial L}{\partial x^i}\frac{d \gamma^i}{ds} + \frac{\partial L}{\partial {\dot{x}^i}}\frac{d \dot{\gamma}^i}{ds}\\
&=\left(\frac{d}{dt}\frac{\partial L}{\partial \dot{x}^i}\right)\frac{d\gamma^i}{ds} +  \frac{\partial L}{\partial {\dot{x}^i}}\frac{d \dot{\gamma}^i}{ds}\\
&=\frac{d}{dt}\left(\frac{\partial L}{\partial \dot{x}^i}\frac{d\gamma^i}{ds}\right)
\end{align*}
So the left-hand side of \eqref{firstvariation} is constant along the curve $\gamma(t)$, as required.

Conversely, a Lagrangian can be associated to any causal geometry, up to a certain ambiguity.  Let $L : T\mathbb{M}\to\mathbb{R}$ be a function homogeneous of degree two that vanishes at every point of the cone over the indicatrix curve $C_P$ for every point $P\in\mathbb{M}$, and suppose further that the vertical Hessian of $L$ is nondegenerate.  It is always possible to select such a function, at least locally, and the ambiguity in selecting such a function is of the form
$$L(x,\dot{x}) \to \Omega(x,\dot{x})L(x,\dot{x})$$
where $\Omega$ is a nonvanishing function.  A causal geometry is already associated to such a Lagrangian, by the preceding argument, and so it is sufficient to show that the causal geometry induced by the Lagrangian agrees with the one already given.  To show that the curves $N_P$ coincide with the null geodesics of $L$, it is enough to show that they are critical points for the energy with respect to variations $\gamma_s$ tangent to the null cone, meaning that \eqref{firstvariation} holds.  Indeed, the Euler--Lagrange system for the constrained problem is
$$\frac{\partial L}{\partial x^i} - \frac{d}{dt}\frac{\partial L}{\partial\dot{x}^i} = \lambda\frac{\partial L}{\partial\dot{x}^i}$$
where $\lambda=\lambda(t)$ is a free function of $t$.  The freedom in $\lambda$ can be absorbed by a change in parameterization of the curve $\gamma(t)$.

Now consider the curve $\gamma = N_{PQ}$, with some parameterization.  Let $\gamma_s$ be a variation of $\gamma$ obtained by varying the endpoint $Q$ along $N_P$ (that is, $\gamma_s$ is a variation through null geodesics).  In the statement of Axiom 5, taking the point $R$ to approach $Q$, a necessary condition for $\gamma_s$ to be on the null cone is that \eqref{firstvariation} holds.  Moreover, also by Axiom 5, $\gamma_s$ satisfies $L(\gamma_s,\dot{\gamma}_s)\equiv 0$.  Therefore we have on the one hand
$$0=\frac{d}{ds}L(\gamma_s,\dot{\gamma}_s) = \frac{\partial L}{\partial x^i} \frac{d\gamma^i}{ds}+\frac{\partial L}{\partial\dot{x}^i}\frac{\dot{\gamma}^i}{ds}.$$
And on the other hand also
$$0 =\frac{d}{dt}\left(\frac{\partial L}{\partial\dot{x}^i}\frac{d\gamma^i}{ds}\right) =\left(\frac{d}{dt}\frac{\partial L}{\partial \dot{x}^i}\right)\frac{d\gamma^i}{ds} +  \frac{\partial L}{\partial {\dot{x}^i}}\frac{d \dot{\gamma}^i}{ds}.$$
Equating the two gives
$$\left(\frac{\partial L}{\partial x^i} - \frac{d}{dt}\frac{\partial L}{\partial \dot{x}^i}\right)\frac{d\gamma^i}{ds} = 0$$
along $N_{PQ}$.  Thus $N_{PQ}$ is a null geodesic for the Lagrangian $L$.


These considerations are summarized in the following theorem:
\begin{theorem}
Let $L:T\mathbb{M}\to\mathbb{R}$ be a homogeneous function of degree two, such that for each $P\in\mathbb{M}$, the curve $C_P = \{L=0\}\subset \mathbb{P}T_P\mathbb{M}$ is smooth.  Suppose also that the vertical Hessian of $L$ is nondegenerate at each point of $C_P$.  Then the null geodesic curves of $L$ form a causal geometry.  Every causal geometry arises (locally) in this manner.
\end{theorem}

\subsection{The causal geometry associated to a differential equation}\label{Lagrangian}
At each point $P$ of $\mathbb{M}$, denote by
$$y=f(x;P)$$
the solution to the differential equation represented by $P$.  We break contact invariance in specifying the differential equation, but shall ultimately be concerned only with the contact-invariant information that can be extracted from $f$.  Locally, this solution depends smoothly on $x$ and $P$.  The underlying assumption under which $\mathbb{M}$ is a differentiable manifold smoothly parameterizing a space of solutions is that in local coordinates $P=(p_1,p_2,p_3)$ on $\mathbb{M}$, the Wronskian determinant
\begin{equation}\label{Wronskian}
\left|\begin{matrix}
\frac{\partial f}{\partial p_1} & \frac{\partial f}{\partial p_2}&\frac{\partial f}{\partial p_3}\\
&&\\
\frac{\partial f_x}{\partial p_1} & \frac{\partial f_x}{\partial p_2}&\frac{\partial f_x}{\partial p_3}\\
&&\\
\frac{\partial f_{xx}}{\partial p_1} & \frac{\partial f_{xx}}{\partial p_2}&\frac{\partial f_{xx}}{\partial p_3}
\end{matrix}
\right|\not= 0.
\end{equation}

To obtain a more concrete description of the cone $N_P$ through a particular solution $P\in\mathbb{M}$ and the associated curves $N_{PQ}$ that rule the cone defined in the previous section, suppose that $\gamma(t)$ is a parameterization of the curve $N_{PQ}$ such that $\gamma(0)=P$. Since all points along $\gamma(t)$ are incident with $P$ at the same point $(x,y,y')$, the following two equations must hold along $\gamma$:
\begin{equation}\label{NullCone}
\begin{aligned}
f(x;\gamma(t)) &= f(x;\gamma(0))\\
f_x(x;\gamma(t)) &= f_x(x;\gamma(0)).
\end{aligned}
\end{equation}
Under generic conditions, the second equation can be used to solve for $x$ in terms of $\gamma(0)$ and $\gamma(t)$, and then substituted into the first equation which may then be solved for the admissible values $\gamma(t)$.  Although ostensibly this system of equations is not invariant under contact transformations, by B\"{a}cklund's theorem a contact transformation will preserve the space of solutions $\gamma(t)$.  This is geometrically evident because the causal geometry itself is contact-invariant.

Solutions of these equations arise as first integrals of the differentiated forms
\begin{equation}\label{NullConeFlow}
\begin{aligned}
\frac{d}{dt} f(x;\gamma(t)) &= 0\\
\frac{d}{dt} f_x(x;\gamma(t)) &= 0
\end{aligned}
\end{equation}
with initial conditions $\gamma(0), \gamma'(0)$.  The initial conditions are not completely arbitrary.  Rather if $\gamma(0)$ is fixed, then $\gamma'(0)$ is constrained by the requirement that
\begin{equation}\label{curve}
\begin{aligned}
\left.\frac{d}{dt} f(x;\gamma(t))\right|_{t=0} &= 0\\
\left.\frac{d}{dt} f_x(x;\gamma(t))\right|_{t=0} &= 0
\end{aligned}
\end{equation}

The second equation of \eqref{curve} can be solved to obtain the parameter $x$ in terms of the initial conditions $\gamma(0)$, $\gamma'(0)$, by \eqref{Wronskian}.  The following Lagrangian is homogeneous of degree two in $\gamma'(0)$:
\begin{equation}\label{lagrangian}
L(\gamma(0),\gamma'(0)) =\left.\frac{d}{dt} f(x(\gamma(0),\gamma'(0));\gamma(t))\frac{d}{dt} f_{xx}(x(\gamma(0),\gamma'(0));\gamma(t))\right|_{t=0}
\end{equation}
The second factor ensures that the Lagrangian is regular (Lemma \ref{regularlagrangian}).  For fixed $\gamma(0)$, the values of the tangent $\gamma'(0)$ satisfying equation \eqref{curve} cut out an affine cone over a curve in the projective tangent space at $\gamma(0)$.  This cone is the ``null cone'' for the Lagrangian $L$, and it coincides with the tangent cone to $N_{\gamma(0)}$ at the vertex $\gamma(0)$.  As in Axiom 4, for $P\in\mathbb{M}$, denote by $C_P\subset \mathbb{P}T_P\mathbb{M}$ the curve cut out in the projective tangent space by the equation $L(P,v)=0$.

Although the Lagrangian $L$ is not itself contact-invariant, by the argument already given its locus of zeros is contact-invariant. Under contact transformations, the Lagrangian is determined up to rescaling by a nonvanishing function of $\gamma(0)$ and $\gamma'(0)$,
$$L(\gamma(0),\gamma'(0)) \to \Omega(\gamma(0),\gamma'(0)) L(\gamma(0),\gamma'(0)).$$
The curves $N_{PQ}$ that generate the causal cone $N_P$ are extremals for the energy functional
$$E[\gamma] = \frac{1}{2}\int_a^b L(\gamma(t),\gamma'(t))\,dt.$$

\begin{theorem}\label{causalTheorem}
A third-order differential equation determines a causal geometry satisfying Axioms 1--5.
\end{theorem}

\begin{proof}
Axiom 1 reflects our running localization assumption that the third-order differential equation under consideration is regular. In local coordinates on $\mathbb{M}$, the system of equations
\begin{equation}\label{Np}
f(x;q_1,q_2,q_3)=f(x;p_1,p_2,p_3),\quad f_x(x;q_1,q_2,q_3)=f_x(x;p_1,p_2,p_3)
\end{equation}
for unknowns $x,Q=(q_1,q_2,q_3)$ and fixed $P=(p_1,p_2,p_3)$ has Jacobian matrix
$$\begin{bmatrix}
f_x(x;q_1,q_2,q_3)-f_x(x;p_1,p_2,p_3) & \frac{\partial f}{\partial q_1}& \frac{\partial f}{\partial q_2}& \frac{\partial f}{\partial q_3}\\
&&&&\\
f_{xx}(x;q_1,q_2,q_3)-f_{xx}(x;p_1,p_2,p_3) & \frac{\partial f_x}{\partial q_1}& \frac{\partial f_x}{\partial q_2}& \frac{\partial f_x}{\partial q_3}
\end{bmatrix}.$$
The matrix always has rank two by \eqref{Wronskian}.  Moreover, at a solution $Q\not=P$ of the original system, the lower left-hand corner cannot be zero, by uniqueness of solutions.  Thus the first column, along with one of the remaining three columns must yield an invertible $2\times 2$ submatrix.  The implicit function theorem then implies that the solution is a smooth submanifold away from the vertex $Q=P$.

For axiom 2, $N_{PR}$ is the set of solutions $Q=(q_1,q_2,q_3)$ of \eqref{Np} for a given value of $x, P=(p_1,p_2,p_3)$.  Because \eqref{Np} has rank two in the $q_i$ variables, the space of solutions is a smoothly embedded curve.  These curves clearly cover $N_P$.  Furthermore, two distinct curves meet only at the vertex $P$, by uniqueness of solutions of differential equations.

Axiom 3 is obvious.  Axiom 4 is equivalent to the assertion that $L$ is a regular Lagrangian, which is proven in Lemma \ref{regularlagrangian} of the next section.

Finally, suppose without loss of generality that in the statement of axiom 5, $R$ is a point between $P$ and $Q$ on $N_{PQ}$.  Along the curve $N_{PQ}$ the value of the parameter $x$ is fixed, and $N_{PQ}$ itself consists of all points $R$ such that
\begin{align*}
f(x; R) &= f(x; P)\\
f_x(x; R) &= f_x(x; P).
\end{align*}
The tangent plane to $N_P$ at the point $R$ is the annihilator of $d_Rf(x,R)$.\footnote{Here and elsewhere, the notation $d_Rf(x;R)$ is the exterior derivative of $f$ with respect to the variable $R$ only.  Equivalently, it is the exterior derivative of $f$ modulo the relation $dx=0$.}  This is the same tangent plane as that obtained by interchanging the roles of $P$ and $Q$.
\end{proof}

\subsection{Associated Hamiltonian}\label{Hamiltonian}
As in the previous section, for $P\in\mathbb{M}$, let $y=f(x;P)$ denote the solution of the differential equation corresponding to $P$.  Once again, contact-invariance is broken, but ultimately we will only be concerned with contact-invariant information contained in the solution.  Let $d_Pf$ denote the exterior derivative of $f$ regarding $x$ as constant.  In local coordinates $(p_1,p_2,p_3)$ on $\mathbb{M}$,
$$d_Pf(x;p_1,p_2,p_3) = \frac{\partial f}{\partial p_1}dp_1+\frac{\partial f}{\partial p_2}dp_2+\frac{\partial f}{\partial p_3}dp_3.$$
Then $x\mapsto d_Pf(x;P)$ defines a curve in the cotangent space $T_P^*\mathbb{M}$.  Denote the associated projective curve by $\tilde{C}_P\subset \mathbb{P}T^*\mathbb{M}$.

This curve is linked to the curve $C_P$ cut out by the Lagrangian via the following construction. Let $V$ be a three-dimensional vector space and $C$ a smooth curve in the projective plane $\mathbb{P}V$.  The dual projective plane $\mathbb{P}V^*$ is naturally identified with the space of lines in $\mathbb{P}V$.  The dual curve $C^*$ is the curve in $\mathbb{P}V^*$ defined by locus of lines tangent to $C$.  Suppose that $C$ is a nondegenerate curve in $\mathbb{P}V$, with parameterization $t\mapsto \gamma(t)$.  At a point $\gamma(t)$ of $C$, the corresponding point of the dual curve is obtained by solving for $\gamma^*(t)\in \mathbb{P}V^*$ the equations
\begin{align*}
\langle \gamma^*(t), \gamma(t)\rangle &= 0 \\
\langle \gamma^*(t), \gamma'(t)\rangle &= 0.
\end{align*}
Properly speaking, to make sense of the second equation, it is necessary to choose a lift of $\gamma$ to a curve in $V$.  Modulo the first equation, the second equation does not depend on the choice of lift, and so there is no ambiguity in speaking of the solution of the system of equations.

\begin{proposition}\label{dualcurve}
$\tilde{C}_P \subset \mathbb{P}T^*_P\mathbb{M}$ and $C_P\subset \mathbb{P}T_P\mathbb{M}$ are mutually dual.
\end{proposition}

\begin{proof}
The curve $\tilde{C}_P$ is characterized as the image of the map
$$x\mapsto d_Pf(x;P).$$
The dual curve to $\tilde{C}_P$ is defined by the equations
\begin{equation}\label{CPequations}
\begin{aligned}
\langle \gamma^*(x), d_Pf(x;P)\rangle &= 0 \\
\langle \gamma^*(x), d_Pf_x(x;P)\rangle &= 0.
\end{aligned}
\end{equation}
But these two equations are identical with the equations that characterize $C_P$.
\end{proof}

Alternatively, choose local coordinates at $P$ and linearize the differential equation at the solution defined by $P$.  Then
$$f(x;p_1,p_2,p_3)=\phi_1(x)p_1+\phi_2(x)p_2+\phi_3(x)p_3$$
and so
$$d_Pf =\phi_1(x)dp_1+\phi_2(x)dp_2+\phi_3(x)dp_3.$$
The incidence relation between $f(x;p_1,p_2,p_3)$ and a nearby solution $f(x;p_1+dp_1,p_2+dp_2,p_3+dp_3)$ is then precisely
\begin{align*}
\phi_1(x)dp_1 + \phi_2(x)dp_2 + \phi_3(x)dp_3 &=0\\
\phi'_1(x)dp_1 + \phi'_2(x)dp_2 + \phi'_3(x)dp_3 &=0
\end{align*}
but these are the same equations that characterize the dual curve of $\tilde{C}_P$.

\begin{lemma}\label{regularlagrangian}
Let $L:T\mathbb{M}\to\mathbb{R}$ be a Lagrangian \eqref{lagrangian}.  Then $L$ is regular in a neighborhood of each point of $C_P$.
\end{lemma}

Indeed, it is sufficient to show that the Hessian matrix $\partial^2 L/\partial \dot{q}_i\partial \dot{q}_j$ is nonsingular at each point of $C_P$.  The Lagrangian is defined by
\begin{equation}\label{lagrangian}
L(q,\dot{q}) = \left(\dot{q}_i\frac{\partial f}{\partial q_i}(x(q,\dot{q}),q)\right)\left(\dot{q}_j\frac{\partial f_{xx}}{\partial q_j}(x(q,\dot{q}),q)\right).
\end{equation}
The function $x(q,\dot{q})$ is defined by
\begin{equation}\label{xqq}
\dot{q}_i \frac{\partial f_x}{\partial q_i}(x(q,\dot{q}),q) = 0.
\end{equation}
In particular, by implicit differentiation,
\begin{equation}\label{implicitx}
\frac{\partial x}{\partial \dot{q}_i} = -\frac{\partial f_x/\partial q_i}{\dot{q}_k\partial f_{xx}/\partial q_k}.
\end{equation}
At a point of $C_P$, in addition the following holds:
\begin{equation}\label{xCP}
\dot{q}_i \frac{\partial f}{\partial q_i}(x(q,\dot{q}),q) = 0.
\end{equation}

The Hessian at a point of $C_P$ is computed by differentiating \eqref{lagrangian}, imposing \eqref{xqq} and \eqref{xCP} along the way, and then finally substituting \eqref{implicitx}. In detail, denote $f_i=\partial f/\partial q_i$, $f_{xi}=\partial f_x/\partial q_i$, etc., and $x_i=\partial x/\partial \dot{q}_i$.  Then
\begin{align*}
L &= \dot{q}^k f_k \dot{q}^\ell f_{xx\ell}\\ 
L_i &= f_i \dot{q}^\ell f_{xx\ell} + \dot{q}^k f_k f_{xxi} + \dot{q}^k f_{xk} x_i \dot{q}^\ell f_{xx\ell} + \dot{q}^kf_k\dot{q}^\ell f_{xxx\ell}x_i\\
L_{ij} &= f_{xi}x_j\dot{q}^\ell f_{xx\ell} + f_i f_{xxj} + f_i\dot{q}^\ell f_{xxx\ell}x_j + f_j f_{xxi} +\\
&\qquad f_{xj}x_i\dot{q}^\ell f_{xx\ell} + \dot{q}^kf_{xxk}x_ix_j\dot{q}^\ell f_{xx\ell} + f_j\dot{q}^\ell f_{xxx\ell}x_i \qquad\pmod{\eqref{xqq}, \eqref{xCP}}\\
&= -f_{xi}f_{xj} + f_i f_{xxj} + f_{xxi} f_j - f_i f_{xj}\left(\frac{\dot{q}^\ell f_{xxx\ell}}{\dot{q}^kf_{xxk}}\right)  - f_{xi}f_j\left(\frac{\dot{q}^\ell f_{xxx\ell}}{\dot{q}^kf_{xxk}}\right)
\end{align*}
by \eqref{implicitx}.  Thus the Hessian matrix of $L$ has the form
$$\operatorname{Hess} L = -AA^T + BC^T +CB^T$$
where the column vectors $A,B,C$ are defined by
\begin{align*}
A_i &= f_{xi}\\
B_i &= f_i\\
C_i &= f_{xxi} - \frac{\dot{q}^\ell f_{xxx\ell}}{\dot{q}^kf_{xxk}}f_{xi}.
\end{align*}
By the hypothesis \eqref{Wronskian}, $A,B,C$ are linearly independent, and so
$$\det \operatorname{Hess} L = -(\det [A\ B\ C])^3\not=0,$$
which establishes the lemma.

The Legendre transformation associated to the Lagrangian $L$ is a function $\mathscr{L} : T\mathbb{M}\to T^*\mathbb{M}$ covering the projection onto $M$.  Over a point $P\in\mathbb{M}$, $L:T_P\mathbb{M}\to\mathbb{R}$, and $\mathscr{L} = d^\vee L$, the vertical exterior derivative of $L$.  The Hamiltonian associated to the degree $2$ homogeneous Lagrangian $L$ is defined by $H = L\circ \mathscr{L}^{-1}$.  Here $\mathscr{L}^{-1}$ is the inverse, possibly only locally defined near points of $C_P$, of the function $\mathscr{L} : T_PM\to T_P^*M$.

By Lemma \ref{regularlagrangian}, $H$ is well-defined in a neighborhood of the preimage of $C_P$ under $\mathscr{L}$.  The Hamiltonian, where it is defined, vanishes precisely on the dual curve $\tilde{C}_P$.  Indeed, in local coordinates,
\begin{align*}
\mathscr{L}_i &= \frac{\partial L}{\partial \dot{q}^i} = f_i\dot{q}^\ell f_{xx\ell}+\dot{q}^kf_kf_{xxi} + \dot{q}^kf_{xk}x_i\dot{q}^\ell f_{xx\ell}+\dot{q}^kf_k\dot{q}^\ell f_{xxx\ell}\\
&=f_i\dot{q}^\ell f_{xx\ell}
\end{align*}
when evaluated at any point of $C_P$.  Thus $\mathscr{L}$ is proportional to $d_p f$ at each point of $C_P$.  Inverting, we conclude that $L\circ\mathscr{L}^{-1}(d_Pf)=0$, so $H$ vanishes along $\tilde{C}_P$.

The following alternative construction of the Hamiltonian also applies, by linearizing the problem at $P$.  In local coordinates at $P$, $\phi_i(x) = \frac{\partial f}{\partial p_i}$, $i=1,2,3$ define independent solutions of the linearized ordinary differential equation.  In terms of these three solutions, the linearized equation itself can be recovered by solving the $3\times 3$ system for the unknown coefficients $h_i$:
\begin{equation}\label{linearizedode}
\phi_i(x) h_0(x) + \phi_i'(x) h_1(x) + \phi_i''(x) h_2(x) = \phi_i'''(x),\quad i=1,2,3.
\end{equation}
Cramer's rule gives
$$h_0(x) = \frac{
\left|\begin{matrix}
\phi_1''' & \phi_1' & \phi_1''\\
\phi_2''' & \phi_2' & \phi_2''\\
\phi_3''' & \phi_3' & \phi_3''
\end{matrix}\right|
}{
\left|\begin{matrix}
\phi_1 & \phi_1' & \phi_1''\\
\phi_2 & \phi_2' & \phi_2''\\
\phi_3 & \phi_3' & \phi_3''
\end{matrix}\right|
},\qquad 
h_1(x) = \frac{
\left|\begin{matrix}
\phi_1 & \phi_1''' & \phi_1''\\
\phi_2 & \phi_2''' & \phi_2''\\
\phi_3 & \phi_3''' & \phi_3''
\end{matrix}\right|
}{
\left|\begin{matrix}
\phi_1 & \phi_1' & \phi_1''\\
\phi_2 & \phi_2' & \phi_2''\\
\phi_3 & \phi_3' & \phi_3''
\end{matrix}\right|
}
$$

$$h_2(x) = \frac{
\left|\begin{matrix}
\phi_1 & \phi_1' & \phi_1'''\\
\phi_2 & \phi_2' & \phi_2'''\\
\phi_3 & \phi_3' & \phi_3'''
\end{matrix}\right|
}{
\left|\begin{matrix}
\phi_1 & \phi_1' & \phi_1''\\
\phi_2 & \phi_2' & \phi_2''\\
\phi_3 & \phi_3' & \phi_3''
\end{matrix}\right|
}.$$

The Lagrangian of the original equation localized at the point $P$ is equal to the Lagrangian of the linearized equation.  It is given first by solving
$$\phi_1'(x) q_1 + \phi_2'(x) q_2 + \phi_3'(x) q_3 = 0$$
for $x$ as a function of $q_1,q_2,q_3$.  In that case,
$$L(q) = \left(\phi_1(x(q)) q_1 + \phi_2(x(q)) q_2 + \phi_3(x(q)) q_3 \right)\left(\phi_1''(x(q)) q_1 + \phi_2''(x(q)) q_2 + \phi_3''(x(q)) q_3 \right).$$

The associated Hamiltonian is obtained by the same construction, but applied to solutions $\tilde{\phi}_1, \tilde{\phi}_2, \tilde{\phi}_3$ of the adjoint equation to \eqref{linearizedode}:
$$y(x) h_0(x) - (y(x)h_1(x))' + (y(x)h_2(x))'' = - y'''(x).$$

The following theorem is due to Wilczynski \cite{Wilczynski}; cf. also Olver \cite{Olver}:
\begin{theorem}
The projective curves $x\mapsto \phi_1(x)\,\partial/\partial q_1 + \phi_2(x)\,\partial/\partial q_2 + \phi_3(x)\,\partial/\partial q_3$ in $\mathbb{P}T_P\mathbb{M}$ and $x\mapsto \tilde{\phi}_1(x)\,dq_1+\tilde{\phi}_2(x)\,dq_2+\tilde{\phi}_3(x)\,dq_3$ are mutually dual.
\end{theorem}


\section{Hamiltonian spray}
As the point $P$ varies, the dual curve $\tilde{C}_P$ cut out by the Hamiltonian defines a subfibration of the projective cotangent bundle $\mathbb{P}T^*\mathbb{M}$.  The four-manifold defined by the total space of this fibration is denoted here by $\mathbb{S}$, and the cone over $\mathbb{S}$ in the cotangent bundle $T^*\mathbb{M}$ is denoted by $\tilde{\mathbb{S}}$.  The canonical one-form $\theta$ on $T^*\mathbb{M}$ is homogeneous of degree one, and pulls back to a natural one-form defined up to scale on $\tilde{\mathbb{S}}$.

For a fixed choice of Hamiltonian homogeneous of degree two, define the Hamiltonian spray on $T^*\mathbb{M}$ as the unique vector field $X$ such that
$$X\lrcorner d\theta = dH.$$
The Hamiltonian spray is invariant up to scale under rescalings of $H$.  Moreover, it is tangent to the variety $\tilde{\mathbb{S}}$ cut out by $H=0$ since $X\lrcorner dH = 0.$  The vector field $X$ is homogeneous of degree one: if $\mu_t : T^*\mathbb{M}\to T^*\mathbb{M}$ denotes the dilation mapping in the fibers, then $X_{t\alpha} = t(\mu_t)_*X_\alpha.$

The Hamiltonian spray will now be used to define a filtration on $T\mathbb{S}$ in a manner analogous to the proof of Theorem \ref{conformalLorentzian}.  Let $V$ be a nonvanishing vector field that is vertical for the fibration $\mathbb{S}\to\mathbb{M}$.  Let $T^1\subset T\mathbb{S}$ be the subbundle spanned by $X$, $T^2$ the subbundle spanned by $X,V$.  Let $T^3\subset T\mathbb{S}$ be the annihilator of $\theta$.  Since $X$ and $V$ both annihilate $\theta$, $T^2\subset T^3$.  Moreover, since $\theta$ is annihilated by $X$ and is Lie derived along it, $[\Gamma(T^1),\Gamma(T^2)]=\Gamma(T^3)$.  Because $\theta\wedge d\theta$ does not vanish when pulled back along on any section of the cotangent bundle, it also does not vanish on $\mathbb{S}$ and so $T^3$ is not Frobenius integrable at any point, and thus $[\Gamma(T^3),\Gamma(T^3)]=\Gamma(T\mathbb{S})$

It remains only to show that $[\Gamma(T^1),\Gamma(T^2)]=\Gamma(T^3)$.  It is sufficient to prove that $X,V,[X,V]$ are linearly independent.  As in the proof of Theorem \ref{conformalLorentzian}, it is convenient to work with a particular choice of vector field $V$. Define the tensor $h\in \operatorname{Sym}^2 T(T^*\mathbb{M})$ to be the vertical Hessian of $H$.  In coordinates,
$$h^{ij} = \frac{\partial^2H}{\partial p_i\partial p_j},$$
and let $h_{ij}$ be the inverse of $h^{ij}$.  These two tensors can be used to raise and lower indices. Let $\epsilon$ be the associated volume tensor in the fiber.  In coordinates
$$\epsilon = \epsilon^{ijk}dp_i\otimes dp_j\otimes dp_k.$$
The angular momentum
$$L^i = \epsilon^{ijk} p_j h_{k\ell}\frac{\partial}{\partial p_\ell}$$
kills $H$, since $\partial H/\partial p_\ell = p^\ell$, and so is tangent to the null cone bundle $\widetilde{\mathbb{S}}$.  Furthermore, on homogeneous functions of degree zero along the null cone, $L^i$ factors through a scalar operator $L^i = p^i V$, because $p^i$ and $\partial/\partial p_i$ are an orthogonal basis for the orthogonal complement of $p_i$.  This defines a vertical vector $V$.  To show that $X,V,[X,V]$ are linearly independent, it is enough to show that $d\theta([X,L^i],L^j)\not=0$.  The Lie bracket is given by
$$[X,L^i] = - \epsilon^{ijk} p_j \frac{\partial}{\partial x^k} + T^i_j\frac{\partial}{\partial p_j}$$
for some tensor $T$.  So
$$d\theta([X,L^i],L^j) =  \epsilon^{imn} p_m h_{kn} \epsilon^{jkn} p_k = p^ip^j - p^kp_kh^{ij}$$
which does not vanish when restricted to the null cone.

Thus by Lemma \ref{J2characterization}, $\mathbb{S}$ equipped with its geodesic spray and vertical vector field is locally isomorphic to the space $J^2$.  Because of the preferred direction $V$, the inclusion of bundles $T^1\subset T^2$ splits, and thus gives rise to a third-order differential equation the distinguished curves of which are the fibers of $\mathbb{S}\to\mathbb{M}$.  The entire procedure is reversible, by the construction of the preceding section, which establishes Theorem \ref{locallyisomorphic}.

\section{Recovering the degenerate metric}\label{Metric}
The degenerate metric on $\mathbb{S}$ is defined as follows.  A point of $\mathbb{S}$ consists of a point $P\in\mathbb{M}$ and $v\in C_P$.  Now, through $v\in C_P$, there is a uniquely defined {\em osculating conic} to $C_P$ at $v$.  This osculating conic in turn defines a unique conformal metric $h_{P,v} : T^*\mathbb{M}\times T^*\mathbb{M} \to \mathbb{R}$.  A degenerate conformal Lorentzian metric is defined by pullback on the subspace of the cotangent bundle of $\mathbb{S}$ that annihilates the vertical direction:
$$g_{P,v}(\alpha,\beta) = h_{P,v}(\pi_* \alpha,\pi_* \beta).$$
A proper degenerate conformal Lorentzian metric is obtained by dualizing.\footnote{If $V$ is a vector space and $W\subset V$, and $B$ is a nondegenerate bilinear form on $W$, then $B$ gives rise to a linear isomorphism $T_B : W\to W'$.  The dual (degenerate) form on $V'$ is given by the mapping $T_{\tilde{B}} : V' \to V'/W^\perp \xrightarrow{\cong} W' \xrightarrow{T_B^{-1}} W\xrightarrow{\subset} V$ where the first is the quotient map, the second is the natural isomorphism, the third is the inverse of $T_B$, and the last is the inclusion map.} In order to derive the formula for the metric \eqref{Chernmetric}, it is necessary to obtain explicit formulas for the osculating conic of a projective curve.  The overall program is inspired by the work of Wilczynski \cite{Wilczynski}.

Suppose that $\mathbf{\phi}(t) = (\phi_1(t),\phi_2(t),\phi_3(t))$ parametrically specifies the homogeneous coordinates of a projective curve, with $\det(\phi,\phi',\phi'')\not=0$..  The projective curve is a conic provided that there exists a $3\times 3$ symmetric non-singular matrix $A$ such that
$$\mathbf{\phi}^T A \mathbf{\phi} = 0.$$
Supposing that $\mathbf{\phi}$ is given, the task is to determine a matrix $A$ such that at a given point $t=t_0$ this holds to as many orders in the expansion in powers of $t-t_0$ as possible.  Since $A$ is regarded projectively, it has $5$ independent numerical components.  These are obtained by solving the system of $5$ equations linear in the entries of $A$:
\begin{align*}
(\phi^T A \phi)(t_0) &= 0\\
(\phi^T A \phi)'(t_0) &= 0\\
(\phi^T A \phi)''(t_0) &= 0\\
(\phi^T A \phi)'''(t_0) &= 0\\
(\phi^T A \phi)^{(4)}(t_0) &= 0.
\end{align*}
Once such a matrix is found, the obstruction to continuing to the fifth order is the derivative $(\phi^T A \phi)^{(5)}(t_0)$, and is the projective curvature associated to the curve.

By the first two equations, $A\phi(t_0)$ is proportional to the cross product $\phi(t_0)\times\phi'(t_0)$, and since $A$ is taken projectively, we can fix a scale by taking
$$A\phi(t_0) = \frac{\phi(t_0)\times\phi'(t_0)}{\det(\phi(t_0),\phi'(t_0),\phi''(t_0))}$$
or, equivalently, $\phi(t_0)^TA\phi''(t_0)=1$.  The third and fourth equations then give, respectively
\begin{align*}
\phi'^T(t_0)A\phi'(t_0) &= -1\\
\phi'^T(t_0)A\phi''(t_0) &= -\frac{1}{3}\frac{\det(\phi,\phi',\phi''')(t_0)}{\det(\phi,\phi',\phi'')(t_0)}.
\end{align*}
The final equation now gives
$$3\phi''^T(t_0)A\phi''(t_0) + 4\phi'^T(t_0)A\phi'''(t_0) = -\frac{\det(\phi,\phi',\phi^{(4)})(t_0)}{\det(\phi,\phi',\phi'')(t_0)}.$$

Now the coordinates of the curve $\phi(t)$ satisfy a third-order differential equation
$$\phi'''(t)  = h_0(t)\phi(t) + h_1(t)\phi'(t) + h_2(t)\phi''(t)$$
where $h_i$ are given explicitly in terms of determinants of $\phi$ and its first three derivatives as in  \eqref{linearizedode}.  The above equations reduce to
\begin{align*}
\phi^T(t_0)A\phi(t_0) &= 0\\
\phi^T(t_0)A\phi'(t_0) &= 0\\
\phi^T(t_0)A\phi''(t_0) &= 1\\
\phi'^T(t_0)A\phi'(t_0) &= -1\\
\phi'^T(t_0)A\phi''(t_0) &= -\frac{1}{3}h_2(t_0)\\
3\phi''^T(t_0)A\phi''(t_0) + 4\phi'^T(t_0)A\phi'''(t_0) &= -h_2^2(t_0)-h_2'(t_0)-h_1(t_0).
\end{align*}
The last equation simplifies by substituting \eqref{linearizedode} for $\phi'''$ and then using the remaining equations to give
$$3\phi''^T(t_0)A\phi''(t_0) = \frac{1}{3}h_2^2(t_0)-h_2'(t_0)+3h_1(t_0).$$
The obstruction $(\phi^TA\phi)^{(5)}(t_0)$ can now be calculated by expanding any terms involving $\phi''', \phi^{(4)}, \phi^{(5)}$ in terms of lower order and then using the above equations.  We find that
\begin{equation}\label{projectivecurvature}
(\phi^TA\phi)^{(5)}(t_0) = 12h_0(t_0) + 4h_1(t_0) + \frac{8}{9}h_2(t_0)^3 - 6h_1'(t_0) -4h_2(t_0)h_2'(t_0) + 2h_2''(t_0),
\end{equation}
which is {\em precisely} the W\"{u}nschmann invariant for the equation \eqref{linearizedode}.

The matrix $A$ obtained from this procedure is also of interest, because it gives the conformal Lorentzian structure.  In the basis $(\phi,\phi',\phi'')$,\footnote{That this basis is ``nonholonomic'' (i.e., nonconstant) is significant for the inverse construction, discussed presently.} the symmetric 2-tensor $A$ is given by
$$A=\begin{bmatrix}
0&0&1\\
&&\\
0&-1&-\frac{1}{3}h_2(t_0)\\
&&\\
1&-\frac{1}{3}h_2(t_0)& \frac{h_2^2(t_0)-3h_2'(t_0)+9h_1(t_0)}{9}
\end{bmatrix}.$$

When, as in section \ref{Hamiltonian}, the curve $\phi(t)\in T_P\mathbb{M}$ is the linearization of the solution $f(x;P)$ to the differential equation at a point $P\in\mathbb{M}$, then \eqref{linearizedode} is the linearization of the differential equation at $P$:
$$h_0(x) = F_y(x,f(x;P),f_x(x;P),f_{xx}(x;P)),\quad  h_1(x) = F_p(x,f(x;P),f_x(x;P),f_{xx}(x;P)),$$
$$h_2(x) = F_q(x,f(x;P),f_x(x;P),f_{xx}(x;P))$$
The projective curvature obtained from \eqref{projectivecurvature} agrees with the W\"{u}nschmann invariant \eqref{Wunschmann} of the original equation. Substituting in for the components of the 2-tensor $A$ gives the metric
$$g = -\left(\frac{\partial}{\partial p}\right)^2 + 2\frac{\partial}{\partial q}\frac{\partial}{\partial y} - \frac{2}{3}F_q\frac{\partial}{\partial p}\frac{\partial}{\partial q} + \frac{F_q^2-3 V(F_q) + 9 F_p}{9}\left(\frac{\partial}{\partial q}\right)^2.$$
The dual degenerate conformal metric, defined on the full tangent bundle of $\mathbb{S}$, agrees with \eqref{Chernmetric}. 

\section{Inverse problems}
By Theorem 1, every causal geometry gives rise to a third-order differential equation.  A more subtle inverse problem is, given a  rank three degenerate conformal Lorentzian metric on the tangent bundle of a four manifold, when is there a causal geometry from which it arises?  More precisely, given a one-parameter family of plane conics, defined by $3\times 3$ nondegenerate symmetric 2-tensors $A(t)$, when is there a plane curve $\phi(t)$ such that, as $t\to t_0$,
$$\phi(t)^TA(t_0)\phi(t) = O(t-t_0)^5?$$
Interchanging $t$ and $t_0$, a necessary and sufficient condition is that
$$\phi(t_0)^TA(t)\phi(t_0) = O(t-t_0)^5.$$
So at each $t_0$, the point $\phi(t_0)$ must satisfy five equations
\begin{align*}
\phi(t_0)^TA(t_0)\phi(t_0) &= 0\\
\phi(t_0)^TA'(t_0)\phi(t_0) &= 0\\
\phi(t_0)^TA''(t_0)\phi(t_0) &= 0\\
\phi(t_0)^TA'''(t_0)\phi(t_0) &= 0\\
\phi(t_0)^TA^{(4)}(t_0)\phi(t_0) &= 0
\end{align*}
In this system, the matrices $A(t_0), A'(t_0), A''(t_0), A'''(t_0), A^{(4)}(t_0)$ should be regarded as given, and the 3-vector $\phi(t_0)$ as unknown homogeneous coordinates.  The system is clearly overdetermined: the two (projective) degrees of freedom in $\phi(t_0)$ must satisfy five equations.  Geometrically the point $\phi(t_0)$ must simultaneously lie on five plane conics, but the intersection of more than two plane conics is generically empty.

The overdetermined system gives rise to a consistency condition on $A$ and its first four derivatives, which we now describe.  In the generic case, we can solve the linear equations for a $3\times 3$ symmetric matrix $X$
\begin{equation}\label{cramdownsystem}
\begin{aligned}
\tr A(t_0)X &= 0\\
\tr A'(t_0)X &= 0\\
\tr A''(t_0)X &= 0\\
\tr A'''(t_0)X &= 0\\
\tr A^{(4)}(t_0)X &= 0.
\end{aligned}
\end{equation}
This can be solved uniquely for $X$, up to scaling, provided the system has rank five.  The consistency condition is then that the solution $X$ has rank one and so splits as an outer product
$$X = \phi(t_0)\phi(t_0)^T.$$
This happens if and only if every $2\times 2$ minor of $X$ vanishes.\footnote{The minors are not independent, however.  The variety in $\mathbb{P}(\operatorname{Sym}^2(\mathbb{R}^3))$ on which the $2\times 2$ minors of a symmetric $3\times 3$ matrix vanish is the well-known Veronese surface, which is not a complete intersection.  Locally it is the zero locus of any three minors coming from distinct rows and columns.}

\subsection{Intermediate cases}
When \eqref{cramdownsystem} has rank one, the W\"{u}nschmann invariant vanishes.  When it has full rank, then it gives rise to a causal curve provided the $2\times 2$ minors of the solution $X$ all vanish.  A calculation done in {\it Mathematica} shows that, for structures coming from third-order equations, these are the only two possibilities: either the system has full rank (and thus nonzero W\"{u}nschmann) or it has rank one (and zero W\"{u}nschmann).

However, {\em a priori} such a system, coming from an arbitrary degenerate conformal Lorentzian structure in four dimensions, can have any rank between $1$ and $5$.  It is interesting to understand why these intermediate cases do not lead to causal curves.

{\em Rank 2.}  Suppose that \eqref{cramdownsystem} has rank two (in an interval around $t_0$).  Then $A''(t) = f(t)A(t) + g(t) A'(t)$ for some functions $f$ and $g$.  The initial matrices $A(t_0), A'(t_0)$ can be brought simultaneously into diagonal form by a transformation of the form
$$A(t_0)\mapsto M^TA(t_0)M,\quad A'(t_0)\mapsto M^TA'(t_0)M.$$
Relative to this fixed initial basis, $A(t)$ and $A'(t)$ remain diagonal throughout the interval of existence.  It is convenient to put 
$$\vec{A}(t) = \begin{bmatrix}A_{11}(t)\\A_{22}(t)\\A_{33}(t)\end{bmatrix},\quad \vec{X}=\begin{bmatrix}X_{11}\\X_{22}\\X_{33}\end{bmatrix}.$$
The two equations of \eqref{cramdownsystem} reduce to
\begin{align*}
\tr A(t)X &= \vec{A}(t)\cdot \vec{X}=0\\
\tr A'(t)X &= A'_{11}(t)X_{11}+A'_{22}(t)X_{22}+A'_{33}(t)X_{33}=0
\end{align*}
Solving:
$$\vec{X} = \vec{A}(t)\times \vec{A}'(t)$$
up to an overall scale. Now if $\phi(t)$ is a curve solving
$$\phi(t)^TA(t)\phi(t) = 0\quad \phi(t)^TA'(t)\phi(t)=0$$
then the entries of $\phi$ must square to the entries of $\vec{X}$, so that
$$\phi(t) = \begin{bmatrix}\pm\sqrt{\vec{A}\times\vec{A}'\cdot e_1}\\\pm\sqrt{\vec{A}\times\vec{A}'\cdot e_2}\\\pm\sqrt{\vec{A}\times\vec{A}'\cdot e_3}\end{bmatrix}.$$
Differentiating gives
$$\phi_i'(t) = \pm\frac{g(t)}{2} \sqrt{\vec{A}\times\vec{A}'\cdot e_i}$$
and so $\phi'$ is proportional to $\phi$.  Thus the range of $\phi(t)$ is a projective point (in the complex sense).\footnote{Indeed, if $\phi(0)$ and $\phi'(0)$ are linearly independent, then $\phi(0)\times\phi'(0)\cdot\phi(t)$ satisfies a second order ode with both initial conditions zero, so $\phi(0)\times\phi'(0)\cdot\phi(t) \equiv 0$ which implies that $\phi(t)$ is constrained to a projective line.  If $\phi(0)$ and $\phi'(0)$ are linearly dependent, then smoothness of dependence on initial conditions gives the result.}  Therefore there are one or no solutions, depending on whether the square roots are all real.

{\em Rank 3 and 4.}  We argue indirectly that, if \eqref{cramdownsystem} has rank 3 or 4 throughout an interval and $\phi(t)$ is a $C^3$ solution in that interval, then $\phi(t)$ parameterizes either a projective point (as in the rank 2 case) or a line (which is degenerate from our point of view).  We were unable to devise a direct argument analogous to the rank 2 case.  In general, because $\phi(t)$ is a three-vector, there must be a non-trivial linear relation between $\phi(t)$ and its first three derivatives.  This must either give a proper third-order equation in an interval, or else there is a linear relation between $\phi,\phi',\phi''$.  In the latter case, $\phi(t)$ does indeed parameterize a line (if it satisfies a second order equation) or a point (if the equation is first order).  In the former case, the coefficients of $A(t)$ can be given in terms of $\phi(t),\phi'(t),\phi''(t)$ and the coefficients of the third-order equation $h_1(t),h_2(t),h_3(t)$, as in the previous section.  But, as already indicated, this implies that the system \eqref{cramdownsystem} has rank either 1 or 5, a contradiction.

\bibliography{thirdorderodes}{}
\bibliographystyle{hep}
\end{document}